\documentclass[english,12pt]{amsart}
\usepackage{amsmath}
\usepackage{amssymb}
\usepackage{amsthm}
\usepackage[T1]{fontenc}
\usepackage{color}
\usepackage{textcomp}
\usepackage[]{enumerate}
\usepackage{tikz}
\usetikzlibrary{decorations.pathreplacing}
\usetikzlibrary{patterns}
\usetikzlibrary{arrows,shapes,positioning}
\usetikzlibrary{decorations.markings}
\tikzstyle arrowstyle=[scale=1]
\tikzstyle directed=[postaction={decorate,decoration={markings,
    mark=at position .65 with {\arrow[arrowstyle]{stealth}}}}]
\tikzstyle reverse directed=[postaction={decorate,decoration={markings,
    mark=at position .65 with {\arrowreversed[arrowstyle]{stealth};}}}]

\usepackage[english]{babel}

\newtheorem{theorem}{Theorem}
\newtheorem{definition}[theorem]{Definition}
\newtheorem{lemma}[theorem]{Lemma}
\newtheorem{proposition}[theorem]{Proposition}
\newtheorem{corollary}[theorem]{Corollary}

\newcommand{\defeq}{\mathrel{\mathop:}=}

\newcommand{\R}{\mathbb{R}}
\newcommand{\N}{\mathbb{N}}
\newcommand{\e}{\varepsilon}

\newcommand{\U}{\bar U}
\newcommand{\I}{\bar I}
\newcommand{\C}{\mathcal C(M)}
\newcommand{\HM}{\mathcal H(M)}

\renewcommand{\O}{\mathcal O}

\newcommand{\T}{\bar T}
\newcommand{\cT}{\mathcal T}
\newcommand{\la}{\lambda}
\makeatletter
\newcommand{\subalign}[1]{%
  \vcenter{%
    \Let@ \restore@math@cr \default@tag
    \baselineskip\fontdimen10 \scriptfont\tw@
    \advance\baselineskip\fontdimen12 \scriptfont\tw@
    \lineskip\thr@@\fontdimen8 \scriptfont\thr@@
    \lineskiplimit\lineskip
    \ialign{\hfil$\m@th\scriptscriptstyle##$&$\m@th\scriptscriptstyle{}##$\crcr
      #1\crcr
    }%
  }
}
\makeatother

\begin{document}
\title[Invariant measures for typical continuous maps]{Invariant measures  for typical continuous maps on manifolds}

\author{Eleonora Catsigeras}

\author{Serge Troubetzkoy}

\address{Instituto de Matem\'atica y Estad\'istica ``Prof.\ Ing.\  Rafael Laguardia'' (IMERL), Universidad de la Rep\'ublica, Av.\ Julio Herrera y Reissig 565, C.P. 11300, Montevideo, Uruguay}
\email{eleonora@fing.edu.uy}
\urladdr{http:/fing.edu.uy/{\lower.7ex\hbox{\~{}}}eleonora}

\address{Aix Marseille Univ, CNRS, Centrale Marseille, I2M, Marseille, France}
\address{postal address: I2M, Luminy, Case 907, F-13288 Marseille Cedex 9, France}
\email{serge.troubetzkoy@univ-amu.fr}
\urladdr{http://www.i2m.univ-amu.fr/perso/serge.troubetzkoy/} \date{}

\thanks{We gratefully acknowledge support of the projects Physeco, MATH-AMSud,   "Sistemas Din\'{a}micos" funded by CSIC of  Universidad de la Rep\'{u}blica (Uruguay).
The project leading to this publication has also received funding from Agencia Nacional de Investigaci\'{o}n e Innovaci\'{o}n (ANII) of Uruguay, Excellence Initiative of Aix-Marseille University - A*MIDEX and Excellence Laboratory Archimedes LabEx (ANR-11-LABX-0033), French ``Investissements d'Avenir'' programmes.}

\begin{abstract} We study the invariant measures of typical $C^0$ maps on compact connected manifolds with or without boundary, and also of typical homeomorphisms. We prove that the weak$^*$ closure of the set of ergodic measures
coincides with the weak$^*$ closure of the set of measures supported on periodic orbits and also coincides with
the set of pseudo-physical measures.  Furthermore, we show that this set has empty interior in the set of invariant measures.
\end{abstract}

\maketitle

\section{Introduction}
In this article we study the structure of the invariant measures for typical continuous maps of a $C^1$ compact, connected manifold $M$ of finite dimension $m \geq 1$, with or without boundary.    This generalizes the work in our previous
article \cite{CT} where we studied the case when $M$ is an interval.   The study of the invariant measures for typical maps was
initiated recently by Abdenur and Andersson \cite{AA}. Previously studies of the dynamics of typical maps have concentrated on the  topological properties; see \cite{AHK,AP,H1,H2,Oprocha&als,O,OU,PP,Y} and the references therein.

Let $\mathcal E_f$ denote the set of ergodic, $f$-invariant Borel probability measures,
$\mbox{Per}_f$ the set of invariant measures supported on a single  periodic orbit, and $\mathcal O_f$ denote the set of pseudo-physical measures for $f$, (see subsection \ref{sec-pseudophysical}
for the definition).
We always have $\mbox{Per}_{f} \subset {\mathcal E}_f$.
Let $\C$ denote the set of continuous maps of $M$ to itself and $\HM$ denote the set of homeomorphisms of $M$ to itself.
 We endow these spaces with the $C^0$ topology and
 say that a family of maps in $\C$ (resp.\ in $\HM$) is typical if it contains a countable intersection of open and dense family.
Our main result  is the following theorem:

\begin{theorem}\label{Theorem-main}
If $f$ is   typical   in $\C$
(resp.\ $\HM$), then
$$ \overline{ \mathcal E_f }  =\overline{ \mbox{Per}_f} =  \mathcal O_f. $$

\end{theorem}

The following questions arise from  Theorem \ref{Theorem-main}.   Are all the invariant measures  in the closure of the ergodic measures? Are  all the invariant measures   pseudo-physical? We prove that the answer is negative for $C^0$-typical maps:

\begin{theorem}\label{Theorem-main2}
If $f$ is  typical  in $\C$  or if $f$ is   typical    in $\HM$, then the set
$  \mathcal O_f $  has empty interior in the set of all $f$-invariant measures.\end{theorem}

In particular an open and dense set of $f$-invariant measures is not in the closure of the ergodic measures.
Theorem \ref{Theorem-main2} implies that the typical behaviour of homeomorphisms on manifolds widely differ from the typical behaviour of $C^1$ diffeomorphisms. In fact,
Gelfert and Kwietniak proved that for $C^1$-typical diffeomorphisms the set of ergodic measures is dense in the space of invariant measures \cite[Theorem 8.1]{Gelfert-K}.

\vspace{.3cm}

The proof of Theorem \ref{Theorem-main} is split  into several pieces; its main components are  Theorems \ref{TheoremPeriodicSRB-l},
 \ref{theoremBigcapClosShr_q&SRBl-l}, \ref{TheoremSRB-l=ClosureErgodic} and
Corollary \ref{corollaryOprocha&als}.
The main components of the proof of Theorem \ref{Theorem-main2} are Theorems \ref{theoremLebesgue-a-e-isInShrinkingInterval} and \ref{TheoremAddedNonUniqErgodic}, Proposition \ref{propositionAA_1}, and Lemma \ref{Lemma-ConvexCombination}.

{The article is organized as follows.  In Section \ref{sec2} we give some background material and prove a technical lemma on the approximation of measures. In Section \ref{sec3} we define the notion of shrinking sets, and
we prove an extension of  a result of Abdenur and Andersson, as well as some consequences of this result.
We prove Theorem \ref{Theorem-main2} in Section \ref{sec4}. Finally,
Sections \ref{sec5} and \ref{sec6} are dedicated to the proof of
 Theorem \ref{Theorem-main}.}

\section{The set-up}\label{sec2}
Let $M$ be a compact, connected, $C^1$ manifold of finite dimension $m \ge 1$, with or without boundary.
  Once  a Riemannian structure is chosen on $M$, it defines a volume measure  which we will refer to as the Lebesgue measure. It also defines a distance,  $dist(\cdot,\cdot)$, between points. We denote  the $C^0$ distance on $C(M)$ by
  $$\rho(f,g) :=\max_{x\in M} dist (f(x), g(x)),$$
 {This metric makes the set $C(M)$ a complete metric space and generates the $C^0$ topology on $C(M)$.
 Similarly, $\HM$ is a complete  metric space with the $C^0$ topology generated by  the distance  $$\rho_{H} :=\max\{\rho(f,g),\rho(f^{-1},g^{-1}) \}.$$
}

The results of Theorems  \ref{Theorem-main} and \ref{Theorem-main2} hold independently of the choice of the Remannian structure on $M$. In fact,
the topological and Borel-measurable properties of the system   depend  on the  metrizable topology of the manifold $M$, but not on the particular choice
of its metric $dist$.  The properties related to  the pseudo-physical measures   depend only on the set of   pseudo-physical measures inside
the space of all the Borel invariant measures. But  this set  is preserved if  we substitute  the previously chosen Lebesgue measure by any finite Borel measure
equivalent to it (see Definition \ref{definitionpseudo-physical}). Hence, the results remain unchanged if we change the choice of  the volume form.

\subsection{Pseudo-physical measures}\label{sec-pseudophysical}

For any point $x \in M$ and $f \in C(M)$,  let  $p\omega(x)$  be the set of the Borel probability measures on $M$
that are the limits in the weak$^*$ topology of the convergent
subsequences of the sequence
$$\left \{  \frac1n \sum_{j=0}^{n-1} \delta_{f^jx} \right\}_{n \in \N}$$
where $\delta_y$ is the Dirac probability measure supported in $y \in M$.

 A measure $\mu$  is called \emph{physical} if the set of those $x \in M$ for which
$p\omega(x) = \mu$ has positive Lebesgue measure. Note that we do not require physical measures to be ergodic.

In this article we consider a generalization of the above definition, introduced in \cite{CE1} and studied in the $C^1$ case in \cite{CE2}, \cite{CCE1}, \cite{CCE2}.
We fix a distance, $d(\cdot,\cdot)$,  in the space of probability measures that endows  the weak$^*$ topology. It is easy to check that the following definition does not depend on the choice of this distance (see \cite{CE1}).

\begin{definition} \em \label{definitionpseudo-physical}
A probability measure $\mu$ is called \em pseudo-physical  \em if for all $\e > 0$ the set $A_\e(\mu) \defeq \{x \in M: d(p\omega(x),\mu) <  \nolinebreak
\e\}$ has positive Lebesgue
measure.  We denote by $\O_f$   the set of  pseudo-physical measures for $f$.

Note that $\O_f$ is always closed and non-empty, and that any pseudo-physical measure is automatically $f$-invariant, and   we do not require a pseudo-physical measure to be ergodic. (In \cite{CE1,CE2,CCE1} pseudo-physical measures were called SRB-like.)
\end{definition}

\subsection{The simplexes }\label{sec-setup}

We consider an   atlas $\{(U_\alpha,\phi_\alpha)\}$ of the manifold $M$, where $U_\alpha$ are open sets whose union covers $M$, and
$\phi_\alpha : U_\alpha \mapsto \R^m$ are $C^1$ diffeomorphisms.

An Euclidean \emph{$k$-simplex} ($0 \le k \le m$) is the convex hull of a finite set $\{x_0, x_1,\dots,x_k\} \subset \R^m$ such that $\{x_j - x_0: 1 \le j \le k\}$ are
linearly independent.

A \em $k$-simplex \em of $M$ is a nonempty compact set $\T$ contained in a chart $(U_{\alpha},\phi_{\alpha})$ such that $\phi_\alpha(\T)$ is  an Eucliean $k$-simplex of $\R^m$.   Throughout the article $T$ will denote the nonempty interior of an $m$-simplex $\overline{T}$.

$M$ is \em triangulable \em if there exits a finite family $\{\T_1, \dots , \T_\ell\}$  of $m$-simplexes (called a {\em triangulation}) such that $T_i \cap T_j = \emptyset$ if $i \ne j$ and $\cup_i \T_i = M$.

It is well known that any compact $C^1$ manifold is triangulable \cite{W},\cite{M}.
For any $\e > 0$ there exists a triangulation of $M$ such that the diameters of all the simplexes $\T_i$ are smaller than $\e$. To prove this it suffices to notice that
an Euclidean $m$-simplex can be decomposed into a finite number of $m$-simplexes of arbitrarily small diameter.

Consider a triangulation $\cT := \cT_m := \{ \T^m_1,\dots, \T^m_{h_m}\}$ of $M$.
Define
$\partial \cT$ to be  the union of the topological boundaries  $\partial \bar T$  of the simplexes $\bar T$ of  $\cT$.

Consider an  $m$-simplex $\T$ of $M$ and its associated chart $(U_\alpha,\phi_\alpha)$.  Let $ c $ be a point in the interior of $ \T$, which we will call \lq\lq centroid\rq\rq. Let $\{x_0, x_1,\dots,x_m\} \subset \R^m$ be the vertices of  $\phi_{\alpha}(\T)$. Up to translation of the coordinates, it is not restrictive to assume that
   $\phi_{\alpha}(c)$ of $\phi_{\alpha}(\T)$ is located at the origin. Suppose $\lambda > 0$ is not too large. We denote $\lambda \T  \subset U_\alpha$ the new simplex such that $\phi_\alpha (\lambda \T)$ is the Euclidean $m$-simplex with
vertices $\{\lambda x_0, \lambda x_1,\dots,  \lambda x_m\} \subset \phi_\alpha(\U_\alpha)$. Clearly, there exists a $\lambda_{0} > 1$ (depending on $\bar T$  and the chosen centroid $c$) such that $\lambda \bar T$
is well defined for all $\lambda \in (0,\lambda_{0})$.

\begin{figure}[h]
\begin{tikzpicture}
\centering
\draw[thin] (-1,-1) -- (1,-1) -- (0,0.73) -- (-1,-1);
\draw[densely dotted] (-0.45,-0.5) -- (0.55,-0.5) -- (0.05,0.36) -- (-0.45,-0.5) ;
\draw[black,fill=black] (0.1,0) circle (.2ex);
\end{tikzpicture}
\caption{The simplex $\bar T$ is drawn with a solid line, while the simplex $\lambda \bar T$ for a $\lambda =1/2$ is drawn with a dotted line. The solid dot   is
the chosen centroid of the triangle located at the point $c \in T = \mbox{int}(\bar T) \subset U_\alpha$.}\label{fig-centroid}
\end{figure}
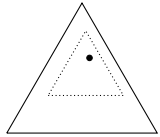

\subsection{A technical lemma}\label{s2}

In the metrizable space of probability measures endowed with the weak$^*$ topology, consider the following distance between probabilities $\mu, \nu$:
$$d(\mu, \nu) := \sum_{i=1}^{\infty} \frac{1}{2^i}\left |    \int \Psi_i \, d\mu - \int \Psi_i \, d\nu \right |, $$
where $\{\Psi_i\}_{i=1}^{\infty}$ is a countable dense family in the space $\mathcal{C}^0(M,[0,1])$.

We will use the following lemma  on the approximation of measures.

\begin{lemma}\label{lemma-dist}
For any   $\e >0 $, there exists  $q \geq 1$ such that,
if $\mu$ and $\nu$ are  probability measures of $M$ satisfying
$supp (\nu) \cup supp (\mu) \subset \cup_{j=1}^m I_j$ for some pairwise disjoint connected closed sets,  if $\mbox{diam}(I_j) \leq 1/q$, and if
$|\nu(I_j) - \mu(I_j) | \le 1/qm$ for each $j$,
then $d(\mu, \nu) < \e$.
\end{lemma}

\begin{proof}
 Fix $n \geq 1$ such that $\sum_{i= n+1} ^{+ \infty} 2^{-i} < {\e}/{2}.$ Then, for any pair of probability measures $\mu, \nu$ we obtain
 \begin{equation}
 \label{eqn01}
 d (\mu, \nu) < \frac{\e}{2} +  \sum_{i= 1}^n   \frac{1}{2^i} \left |    \int_M \Psi_i \, d\mu - \int_M \Psi_i \, d\nu \right |.
 \end{equation}
The uniform continuity of the finite family of functions $\{\Psi_i\}_{1 \leq i \leq n}$, implies that there exists $\delta > 0$ such that, if  $dist(x_1,x_2) < \delta$, then $dist(\Psi_i(x_1) , \Psi_i(x_2))< \e/4$ for all $1 \leq i \leq n$.
Fix $q  \in \mathbb{N}^+$  such that
  $ {1}/{q} <  \min (\delta, \e/4) $.

The mean value theorem for integrals, yields for all $1 \leq i \leq n$
\begin{equation*}
  \begin{split}
   \int_{\bigcup_{j=1}^m I_j}\Psi_i \, d\mu &= \sum _{j= 1} ^m \int_{I_j} \Psi_i  \, d \mu = \sum _{j= 1} ^m\Psi_i(x_j) \mu (I_j) \mbox{ for some } x_j \in I_j,\\
    \int_{\bigcup_{j=1}^m I_j}\Psi_i \, d\nu & = \sum _{j= 1} ^m \int_{I_j} \Psi_i  \, d \nu = \sum _{j= 1} ^m\Psi_i(x'_j) \nu (I_j) \mbox{ for some } x'_j \in I_j.
  \end{split}
  \end{equation*}
\nopagebreak[4]
From this we deduce
$$
 \int_{\bigcup_{j=1}^m I_j}\Psi_i \, d\mu - \int_{\bigcup_{j=1}^m I_j} \Psi_i \, d\nu     = \sum _{j= 1} ^m \left (\Psi_i(x_j) \mu (I_j) -\Psi_i(x'_j) \nu (I_j) \right ) $$ $$
     =\sum _{j= 1} ^m  \left (\Psi_i(x_j) (\mu (I_j) -  \nu (I_j) ) +  (\Psi_i(x_j)   -\Psi_i(x'_j) ) \nu (I_j) \right ).
$$
Since  $|\nu(I_j) - \mu(I_j) | \le 1/qm$ for each $j$, we conclude
$$ \left |\int_M\Psi_i \, d\mu - \int_M \Psi_i \, d\nu \right | \leq \sum _{j= 1} ^m  \frac{1}{qm}    +  \frac{\e}{4} \sum _{j= 1} ^m   \nu (I_j) < \e/2 \ \ \ \forall \ 1 \leq i \leq n.$$
Substituting this inequality  in (\ref{eqn01}), finishes the proof of Lemma \ref{lemma-dist}.
\end{proof}

\section{Shrinking sets}\label{sec3}
 A  \em periodic  shrinking set of period $p \geq 1$,  \em is a nonempty open set $I$ whose closure $\I$ is an $m$-simplex, such that  $\{f^j( \I)\}_{0 \leq j \leq p-1}$ is a family of pairwise disjoint sets,  $f^p(\I) \subset I$, and   $\mbox{diam} (f^j(I) ) <  \mbox{diam}(I)$ for all $1 \leq j \leq p-1$.

An nonempty open set $J$, whose closure $\bar J$ is an $m$-simplex,  is \em eventually periodic  shrinking,  \em if there exists a periodic  shrinking set  $I$ and a {\em transience time}
$n \in \N^+ $ such that  $ f^n(J) \subset I$, and $\mbox{diam}(f^j(J)) <  \mbox{diam}(J)$ for all $1 \leq j \leq n-1$.

Note that the same periodic or eventually periodic  shrinking set   for some $f \in \C$ (resp.\ $\HM$)  is also a periodic or eventually periodic  shrinking set  for all $g$ in a small neighborhood of $f$ in $\C$ (resp.\ $\HM$).

If $I$ is a periodic or eventually periodic shrinking set  with period $p$, then,
by the Brouwer fixed point theorem, the map $f^p$ has a fixed point $x_0 \in I$; hence, the point  $x_0$  is  periodic for $f$  with period $p$.

\begin{theorem} \label{theoremLebesgue-a-e-isInShrinkingInterval}
For a typical map  $f \in \C$ (resp.\ $f \in \HM)$,  Lebesgue a.e.\ point $x \in M$ belongs to a sequence  $\{I_q\}_{q \in {\mathbb N}^+}$ of  eventually periodic  shrinking sets $I_q$ such that   $\mbox{diam}(I_q) < 1/q$.

\end{theorem}

\begin{proof}
For  given natural numbers $q,k \in \mathbb{N}^+$, denote by ${\mathcal S}_{q,k}$  the set of maps in $\C$ (resp.\ $\HM$), for which there exists a  finite family of nonempty open sets,
which we denote by $\{  I_ {1}, \ldots,   I_l\}$,     such that:

\noindent (i)  $diam(I_i) < 1/q$ for   $i= 1, 2, \ldots, l$,

\noindent (ii) $I_i$ is a periodic or eventually periodic shrinking set,

\noindent (iii)  $Leb (M \setminus \bigcup_{i= 1}^l I_i) < 1/k$.

The set ${\mathcal S}_{q,k}$ is open   in $\C$ (resp.\ $\HM$) since, for each $f \in {\mathcal S}_{q,k}$, the same family of such shrinking sets of $f$, is also a family of shrinking sets satisfying conditions (i), (ii), (iii) for any other map $g$ in a sufficiently small neighborhood of $f$.

Now, let us prove that  ${\mathcal S}_{q,k}$ is dense in $\C$ (resp.\ $\HM$). This will complete the proof
of Theorem \ref{theoremLebesgue-a-e-isInShrinkingInterval}, since the set
$${\mathcal S} \defeq \bigcap_{q \ge 1}\bigcap_{k  \ge 1} {\mathcal S}_{q,k}$$
is a dense $G_\delta$-set, and by  construction of ${\mathcal S}_{q,k}$, any map in ${\mathcal S}$ satisfies the conclusion of the  Theorem.

To prove the denseness of ${\mathcal S}_{q,k}$
fix any map $f \in \C$ (resp.\ $\HM)$. For any   $\e >0$, we will construct  a perturbation $g \in {\mathcal S}_{q,k}$ such that $\rho(f,g) <\e$  (resp.\  $\rho_{H} < \e$).

Choose  $0<\delta  < \min\{\e, 1/q\}$  such that $$dist(x,y) < \delta \ \ \Rightarrow \ \ dist(f(x) , f(y))< \frac{\e}{3},$$ and if besides $f \in \HM$, then also $ dist(f^{-1}(x), f^{-1}(y)) < \frac{\e}{3}$.
Consider a triangulation $\mathcal T := \{\overline{T}_1,\dots,\overline{T}_l \}$ of $M$ such that the diameters of all the simplexes $\overline T_i$ are smaller than $\delta$.
For each $i$ we would like to choose a point  $x_i \in T_i = \mbox{int}(\bar T_i)$ such that  $f(x_i) \in T_{j(i)} = \mbox{int}( \bar T_{j(i)})$ for some $j$ which depends on $x_i$.
 If $f$ is an homeomorphism the set of points in $T_i$ whose image is not in $\partial \mathcal T$ is
a non-empty open set, and so we can choose such points $x_i$.  If $f$ is merely continuous it is possible that such points do not exist. Then, we
perturb $f$ to a continuous map $f_1 \in \C$ satisfying this property such that $\rho(f,f_1) < \e/{12}$ and such that
$dist(x ,y) < \delta \ \Rightarrow dist(f_1(x) , f_1(y)) < \e/2$.  For the sake of consistency of notation,
if $f \in \HM$  or if $f \in \C$ already satisfied this requirement we set $f_1 := f$.

Now we define ``spherical coordinates'' on each $\T_i$. To do this consider a chart $(U_\alpha,\phi_\alpha)$ such that $\T_i \subset U_\alpha$.
Consider $S_i := \phi_\alpha(\T_i)$. Up to  a translation of the coordinates, it is not restrictive to assume that $0= \phi_\alpha(x_i)$.
Recall that $x \in T_i = \mbox{int}(\overline T_i)$; hence
      $0 \in \mbox{int}(S_i)$. Therefore, there exists
a small solid $m$-dimensional ball $B(0,r_0)$  lying  inside $S_i$, and the spherical coordinate
system $(r,\theta): r\in [0,r_0], \theta \in \mathbb S^m$ in it.  Using the convexity of $S_i$ we can extend this coordinate system to all of $S_i$ by defining
$R(\theta)$ to be the supremum of all $r$ such that the ray  in the direction $\theta$ lies inside $S_i$ up to the distance $r$.
For simplicity of notation, we normalize the $r$ coordinate, by setting $s := r/R(\theta)$, thus $s \in [0,1]$ for all $\theta$.
Finally we  pull back these
coordinates to $\T_i$ via the map $\phi^{-1}_\alpha$.
For any $i$ and any $s_0 \in (0,1]$ we will denote
$$B^i_S :=  \{(s,\theta) \in \T_i: s \in [0,s_0]  \}.$$

We consider an integer number $n \geq 1$ to be specified later.  We define an homeomorphism $h$ of $M$ as follows.
$h: \T_i \mapsto \T_i$ is given by
  by setting
  $$h(s,\theta) \defeq  ( s^{n}  ,\theta  ).$$
Note that the map $h$ is uniquely defined on $\partial \mathcal T$, because for all $i$ the definition of $h$ on $\partial \T _i$ yields the identity map.
Furthermore, note that since $h$ and $h^{-1}$ map each simplex $\T_i$ to itself, and the diameters of the $\T_i$ are at most $\delta$ we have
$$dist(x,h(x)) \le  \delta \quad \mbox{ and } \quad dist(x,h^{-1}(x)) \le \delta.$$
The   map $g$ is then defined by
$$g = f_1 \circ h.$$
We claim that, for an adequate choice of $n$, the map $g$ has the desired properties, i.e., $g \in {\mathcal S}_{q,k}$ and $\rho(g,f) < \e$ (and $\rho(g^{-1},f^{-1}) < \e$ in the homeomorphism case).

Consider $y \in M$. Since $dist(h(y),y) < \delta$ we have
$$dist(g(y),f_1(y)) = dist(f_1(h(y)),f_1(y)) < \frac{\e}{2}$$
and thus
$$\rho(f,g) < \rho(f,f_1) + \rho(f_1,g) < \e.$$

In the case that $f = f_1 \in \HM$, we have  $g^{-1} = h^{-1} \circ f_1^{-1}$.
Consider $y \in M$, let  $z = f^{-1}(y) = f_1^{-1}(y)$ , then
$$dist(g^{-1}(y),f ^{-1}(y))  =dist(g^{-1}(y),f _1^{-1}(y)) =  dist(h^{-1}(z),z) \le \delta < \e$$
and thus
$$\rho(g^{-1} , f^{-1}) < \e.$$

Now, to end the proof of Theorem \ref{theoremLebesgue-a-e-isInShrinkingInterval}, it is enough to choose $n$ such that the map $g$ above constructed belongs to ${\mathcal S}_{q,k}$.

Choose $x_i$ as the centroid of $T_i$,   choose $\lambda_1 \in (0,1)$ close to 1,  and $\lambda_2 \in (0, \lambda_1)$ close to 0, and construct the simplexes $\lambda_2 \bar T_i \subset \lambda_1 \bar T_i \subset T_i = \mbox{int}(\bar T_i)$  whose interiors contain $x_i$.  If $\lambda_1$ is close enough  1,  we have:
\begin{equation}\label{iii}
Leb\big(M \setminus (\cup_i \lambda_1 \bar T_i)\big) < 1/k.
\end{equation}
Besides, if $\lambda_1 \in (0,1)$ is close enough 1, then $f_1(x_i) \in \mbox{int}(\lambda_1 \bar T_{j(i)})$. So, if $\lambda_2 \in (0,\lambda_1)$ is small enough, we have  $f(x_i) \in f_1(\lambda_2 \T_i) \subset \mbox{int}(\lambda_1 \T_{j(i)}) $ with $\mbox{diam}  (f_1(\lambda_2 \T_i)) <\min_j{\mbox{diam}(\lambda_1 \bar T_j)} $ for all $i$.

We claim that there exists $n$ such that   the sets $\lambda_1 T_i$ are eventually periodic  shrinking sets for $g$. In fact, choose a real number $s_1 > 0$ such that $\lambda_1 \T_i \subset B^i_{s_1}$ and
 choose other real number $s_2 \in (0,s_1)$ such that
 $B^i_{s_2}  \subset \lambda_2 \T_i .$
Finally choose $n$ such that $s_1^{n} < s_2$. With these choices
we have $h(B^i_{s_1}) \subset B^i_{s_2}$
and thus
$$g(\lambda_1 \T_i) = f \circ h (\lambda_1 \T_i) \subset
f \circ h(B^i_{s_1}) \subset f(B^i_{s_2}) \subset f(\lambda_2 \T_i) \subset  {\lambda_1}  T_{j(i)}.$$
Since there exists only a finite number $l$ of simplexes in the triangulation, this assertion implies that $\lambda_1T_i$ is an  eventually periodic shrinking set with the sum of the transience time plus the period bounded by $l$. By construction, Properties (i), (ii) and (iii) are satisfied.   \end{proof}

When  $M$ is a manifold without boundary, the following result was proven by Abdenur and Andersson in \cite{AA} for typical $f \in \C$ in any finite dimension and for typical $f \in \HM$ in dimension at least two.
Our proof works also when $M$ is a manifold with boundary, and also for $f \in \mathcal{H} (M)$ in dimension one.

\begin{corollary}
   \label{theoremAAa)}

\noindent Let $f $ be a typical map in $C(M)$ (resp.\ $\HM$).  Then,
for Lebesgue almost every point $x \in M$, the  sequence $$\Big \{\frac{1} {n}  \sum_{j=0}^{n-1} \delta_{f^j(x)}\Big\} _{n \in \mathbb{N}^+}$$ of empirical probabilities is convergent,
\end{corollary}

\begin{proof}
Consider the families ${\mathcal S}_{q,k}$ for $q,k \in \mathbb{N}^+$, defined at the beginning of the proof of Theorem \ref{theoremLebesgue-a-e-isInShrinkingInterval}. We proved that each ${\mathcal S}_{q,k}$ is open and dense, so typical $f$ belong to $\mathcal S:= \bigcap_{q,k \geq1} {\mathcal S}_{q,k}$. Thus, it is enough that the assertion of this corollary holds for all $f \in {\mathcal S}$.
 Consider a continuous function $\phi: M \to \R$. Since $M$ is compact, for every $\e > 0$ we can find a $\delta > 0$ such that if $dist(x,y) < \delta$ then
 $|\phi(x) - \phi(y)| < \e$.
Suppose $f  \in {\mathcal S}_{q,k}$ with  $q$ such that $1/q < \delta$.  Due to Theorem \ref{theoremLebesgue-a-e-isInShrinkingInterval}, Lebesgue a.e.\ belongs to an eventually periodic shrinking set of diameter smaller than $1/q$.
It is enough to prove the convergence of the corollary for any iterate of $x$ instead of $x$; thus we can assume that $x \in I_i$ where $I_i$ is a periodic   shrinking set.
Denote the period of $I_i$ by $p$. Then for $n > p$ we write $n = \ell p + r$ with $0 \le r < p$. Since $I_i$ is a  periodic shrinking set with diameter smaller than $1/q$ we have
$dist(f^{\ell p + r}x, f^rx) \leq \mbox{diam}f^r(I_i) \leq \mbox{diam}(I_i) < 1/q <\delta$. Thus

\begin{equation}\label{eqcont}\left  |\sum_{j=0}^{\ell p - 1} \phi(f^j x) - \ell \sum_{j=0}^{ p - 1} \phi(f^j x) \right  | < \ell \e.\end{equation}
Then
\begin{equation}\label{eqconv}
\begin{split}
\Big | \frac1n \sum_{j=0}^{n-1} \phi(f^j x) & - \frac{1}{p} \sum_{j=0}^{p-1} \phi(f^j x) \Big | \le \\
& \Big | \frac1n \sum_{j=0}^{n-1} \phi(f^j x)  - \frac{1}{\ell p} \sum_{j=0}^{\ell p-1} \phi(f^j x) \Big | +\\
&\Big | \frac{1}{ \ell p} \sum_{j=0}^{\ell p-1} \phi(f^j x) - \frac1p \sum_{j=0}^{\ell p - 1} \phi(f^j x) \Big |.
\end{split}
\end{equation}

Using  \eqref{eqcont} we see that second term is bounded by $\e/p < \e$.
Using the triangle inequality once more we can bound the first term by
$$ \Big | \frac1n \sum_{j=\ell p}^{n-1} \phi(f^j x)  \Big |+ \Big |  \frac{r}{n \ell p} \sum_{j=0}^{\ell p-1} \phi(f^j x) \Big |. $$
Since $r$ is bounded, for $n$ sufficiently large this is bounded by $\e$, and  the difference in \eqref{eqconv} is bounded by $2\e$.

We deduce that
$$\limsup \frac1n \sum_{j=0}^{n-1} \phi(f^j x)-  \liminf \frac1n \sum_{j=0}^{n-1} \phi(f^j x) < 4\e.$$
Since $\e > 0$ is arbitrary, the result follows.
\end{proof}

A map $f \in \C$ is \em Lebesgue-a.e.~strongly non positively expansive, \em  if for any real number $\alpha >0$,   and for Lebesgue a.e.~$x \in M$,
$$Leb\Big (\{ y \in M : dist(f^n(x), f^n(y))< \alpha \ \  \forall \ n \in  \mathbb{N}  \}   \Big) >0 .$$
So, according to \cite[Definition 2.1]{AM}, the Lebesgue measure is non positively expansive. Moreover, using the notation of  \cite[Section 2]{AM}: $$Leb(\Phi_{\alpha}(x)) \neq 0 \   Leb\mbox{-a.e.  } x \in M, \mbox{ where}$$
$$\Phi_{\alpha}(x):=\{y \in M: \mbox{dist}(f^n(y), f^n(x)) < \alpha.$$ Namely, the positive expansivity of the Lebesgue measure fails not only in some points, but   Lebesgue a.e..

\begin{corollary} \label{corollaryStronglyNonExpansive}
Typical  maps in $\C$ (resp.\ $f \in \HM$)  are Lebesgue-a.e.\ strongly non positively expansive.
\end{corollary}

\begin{proof}

Take any periodic or eventually periodic set $I$  with diameter smaller than $\alpha$. Then,  $\mbox{diam}(f^j(\overline I)) <\alpha$ for all $j \geq 0$. Any two points $x, y \in I$  satisfy $dist(f^j(x), f^j(y)) < \alpha $ for all $j \geq 0$. Thus for any point $x \in I$ we have
$$ Leb\Big (\{ y \in M:  dist(f^n(x), f^n(y))< \alpha \ \  \forall \ n \in  \mathbb{N}  \}   \Big) \geq Leb(I) > 0.$$
Thus the result follows directly from  Theorem \ref{theoremLebesgue-a-e-isInShrinkingInterval}.
\end{proof}

An $f \in \C$ is called \em positively expansive  \em  if there exists a   constant $\alpha >0$,  called the  \em expansivity constant, \em such that, for any two points $x, y \in M$, if  $dist(f^n(x), f^n(y) ) \leq \alpha $ for all $  n \in \mathbb{N},$ then $  x=y.$ Clearly a map which is  Lebesgue a.e.\ strongly non positively expansive is not positively expansive.  For compact connected manifolds with boundary it is known that every map is not positively expansive (Theorem 2.2.19, \cite{AH}). On the other hand, if $M$ has no boundary, then we have examples of positively expansive maps for example hyperbolic expanding toral diffeomorphisms.

\subsection{The sets $AA$ and $AA_1$}

For any  $x \in M$ we define
\begin{equation}
\label{eqn-mu_x}
\mu_x \defeq \lim_{n \to \infty} \frac{1} {n}  \sum_{j=0}^{n-1} \delta_{f^j(x)}\end{equation} when this limit exists; namely, when $p\omega(x) = \{\mu_x\}$.
Let
\begin{equation}
\label{eqnEquationAA}
AA \defeq \{  x \in M: p\omega(x) = \{\mu_x\}\} \mbox{ and }
AA_1 \defeq \{  x \in AA: \mu_x \in \O_f\}.\end{equation}
For typical continuous maps or homeomorphisms,  Corollary \ref{theoremAAa)}    states
\begin{equation}\label{eqAA}
Leb(AA) = Leb(M),
\end{equation}
while   \cite[Proposition 3.1 and Definition 3.2]{CE2} implies that
\begin{equation}\label{eqAA1}
Leb(AA_1) = Leb(M).
\end{equation}
\begin{proposition} \label{propositionAA_1}
If $f $ is typical  in $\C$ (resp.\  $\HM$) then
$$\O_f = \overline{\{ \mu_x: x \in AA_1 \}}.$$
The set $\{x \in AA_1: d(\mu_x, \mu_{x_0}) < \e\}$ has positive Lebesgue measure for every $\e >0$ and for every $x_0 \in AA_1$.
\end{proposition}

\begin{proof}  By definition of the set $AA_1$ we have
$\{\mu_x: x \in AA_1\}  \subset \O_f$. Besides, since $\O_f$ is closed in the weak$^*$ topology (see \cite[Theorem 1.3]{CE1}), we have
 $\overline{\{\mu_x: x \in AA_1\}}  \subset \O_f$.

 We turn to the other inclusion.  Suppose $\mu \in \O_f$, i.e.,
 $$Leb\{ x \in M: d(p\omega(x),\mu) < \e)\} > 0 \text{ for all } \e > 0.$$
Since $Leb(AA_1) =1$ and $p\omega(x) = \mu_x$ for $x \in AA_1$,
this is equivalent to
  $$Leb\{ x \in AA_1: d(\mu_x,\mu) < \e)\} > 0  \text{ for all } \e > 0.$$
Thus $\mu \in \overline{\{ \mu_x: x \in AA_1 \}}$. Since $\mu \in \O_f$ is arbitrary we conclude
 $\O_f \subset \overline{\{ \mu_x: x \in AA_1 \}}.$ The first assertion of Proposition \ref{propositionAA_1} is proven.

 To prove the second assertion recall  the definition of pseudo-physical measure,  $\mu \in {\mathcal O}_f$ if the set $\{x \in M \colon d(p\omega(x), \mu) < \e\}$ has positive Lebesgue measure. Combining this with Equality (\ref{eqAA1}), we deduce that $\{x \in AA_1 \colon d(\mu_x, \mu) < \e\}$ has positive Lebesgue measure, for any $\mu \in {\mathcal O}_f$, in particular for $\mu_{x_0}$ for any $x_0 \in AA_1$.
\end{proof}

\subsection{Approximation of pseudo-physical measures by periodic measures.}

In this subsection we will prove Theorem \ref{TheoremPeriodicSRB-l}:  each pseudo-physical measure can be arbitrarily approximated by atomic measures supported on periodic point. We will use the following lemma.
\begin{lemma}
\label{lemmaInvMeasOnShrinkingIntervals} Let $f \in \C$ (resp.\ $f \in \HM$) and $\mu$ be an $f$-invariant measure. Assume that $I \subset M$ is a periodic shrinking set of period $p$, and denote  $K \defeq  \bigcup_{j= 0}^{p-1} f^j(\overline I).$ Then \begin{equation}
\label{eqnno.}
\mu (f^j(\overline I)) = \mu(\overline I) = \frac{\mu(K)}{p}   \ \ \ \forall \ j  \geq 0. \end{equation}
If $\mu$ is additionally ergodic, then $\mu(K)$ is either zero or one.
\end{lemma}

\begin{proof}
 For any $j \geq 0$ let $j= kp - i$  with $k \in \N^+$ and $1 \leq i \leq p$. Using that $\mu$  is $f$-invariant, that $f^p(\overline I) \subset I$ and that $f^{-j} (f^j (\overline I)) \supset \overline I$ for all $j \geq 0$, we obtain:
 $$\mu (f^j(\overline I)) =  \mu (f^{-j}(f^j(\overline I)) \geq \mu (\overline I) \geq \mu(f^{kp}(\overline I)) = \mu (f^{-i} (f^{kp}(\overline I)) \geq \mu(f^j(\overline I)).$$
Hence, all the inequalities above are equalities; and thus $\mu(f^j(\overline I)) = \mu(\overline I)$ for all $j \geq 0$.
Assertion (\ref{eqnno.}) follows, since $\mu(K) = \sum_{j= 0}^{p-1} \mu ( f^j(\overline I) )$.

Finally, since $f(K) \subset K$, we have $K \subset f^{-1}(f(K)) \subset f^{-1}(K)$.  Since $\mu$ is ergodic, the set $K$  must have $\mu$-measure equal to zero or one.
\end{proof}

 \begin{theorem}\label{TheoremPeriodicSRB-l}
 Typical maps  $f \in \C$ (resp.\ $f \in \HM)$,  satisfy
 ${\mathcal O}_ f \subset \overline{\mbox{Per}_{f}}.$
 \end{theorem}

 \begin{proof}
Choose  $q_n \geq 1$ such that  any two  measures  satisfying the $q_n$-approach conditions of  Lemma \ref{lemma-dist}, are mutually at distance smaller than $1/n$.
Let $\mu \in \O_f$.
By the definition of pseudo-physical measure, the set $A_n \subset M$ of points $y$ such that $$
d(p \omega (y), \mu)< 1/n, $$ has positive Lebesgue measure. If $f$ is typical in $\C$ or $\HM$, then Lebesgue a.e.~$y_n \in A_n$ satisfy $p\omega(y_n) = \{\mu_{y_n}\}$ (Corollary \ref{theoremAAa)}) and is contained in an arbitrarily small eventually periodic shrinking set (Theorem \ref{theoremLebesgue-a-e-isInShrinkingInterval}).
Therefore,  \begin{equation}
\label{eqn09}
d(\mu_{y_n}, \mu)< 1/n, \end{equation} and $\mu_{y_n}\ $ is supported on $K_n \defeq \bigcup_{j= 1}^{p_n} f^j(\overline I_n)$,  where $I_n$ is a periodic shrinking set of period $p_n$,  such that $\mbox{diam}(I_n) \leq 1/q_n$. Applying Lemma \ref{lemmaInvMeasOnShrinkingIntervals} and the definition of shrinking set yields $\mu_{y_n}(f^j(\overline I_n)) = 1/p_n$.

Since the set $I_n$ is shrinking periodic with period $p_n$, there exists at least one periodic orbit of period $p_n$ in $K_n$. Consider the periodic invariant measure $\nu_n $ supported on this periodic orbit. It also satisfies $\nu_n(f^j(\overline I)) = 1/p_n$. So, applying Lemma \ref{lemma-dist} we deduce that
\begin{equation}
\label{eqn10}   \  \exists \   \nu_n \in \mbox{Per}_{f}\colon \
 d(\mu_{y_n}, \nu_n) \leq 1/n. \end{equation}

Combining Inequalities (\ref{eqn09}) and (\ref{eqn10}), we deduce that there exists    sequence of periodic atomic measures $\nu_n$ such that
 $d(\nu_n, \mu) < 2/n$, finishing the proof of Theorem \ref{TheoremPeriodicSRB-l}.
\end{proof}

\section{Most invariant measures are not pseudo-physical}\label{sec4}
In this section we will prove Theorem \ref{Theorem-main2}.
We start with an extension of  Theorem of Abdenur and Andersson in \cite[Theorem 3.6]{AA}:

\begin{theorem}
 {\bf (Extension  of Abdenur-Andersson Theorem) } \label{theoremAA}
\noindent For typical maps in $C(M)$  there does not exist physical measures.
\end{theorem}

\begin{proof}
In \cite[Theorem 3.6]{AA}, Abdenur and Andersson   prove this  assertion  in the case that $M$ is a compact connected
$C^1$ manifold without boundary.
Now, let us prove that it also holds   when $M$ has boundary.   Given a manifold
with boundary $M$, take two disjoint copies $M_1, M_2$ and identify each boundary point in $M_1$ with the same point in $M_2$. We obtain the double cover
$D(M)$ of $M$ which is a $C^1$ manifold  without boundary  with the same dimension as $M$. See for example 9.32 in \cite{L} for details of this consruction.

Suppose that $M$ has boundary, and denote by ${\mathcal V}$ the family of all the continuous map  $f \in \C$ such that $f(M) \subset M \setminus \partial M$.
 By construction ${\mathcal V}$  is   open and dense in $\C$.  For each $f \in {\mathcal V}$, construct the nonempty  set ${\mathcal S} _{f} \subset C(D(M))$ composed by all the continuous maps $g: D(M) \mapsto D(M)$ such that
$g|_{M_1}= f$.

We claim that for any open set ${\mathcal U} \subset {\mathcal V}$ the set
\begin{equation} \label{eqnU'} {\mathcal U}' := \bigcup_{f \in {\mathcal U}} {\mathcal S} _f = \Big \{g \in C(D(M)) \colon  g|_{M_1}=: f \in {\mathcal U} \Big\}\nonumber\end{equation}
 is open in $C(D(M))$;  in particular ${\mathcal V}' $ is open in $C(D(M))$. This claim follows from that fact that
the restriction that sends each map $g \in C(D(M)) $ to the map $g|_{M_1} : M_1 \mapsto g(M_1) \subset M$ is a continuous operator. Therefore, the preimage ${\mathcal U}'$ by that restriction, of the open family ${\mathcal U} \subset {\mathcal V} \subset  C(M) \subset C(M_1, D(M)) $, is open.

 Next we claim that a converse-like statement is true: for any open set ${\mathcal U}' \subset {\mathcal V}' \subset C(D(M))$ the set
  \begin{equation} \label{eqnU} {\mathcal U} :=    \Big \{f \in {\mathcal V} \colon   \ \exists \ g \in {\mathcal U'} \subset C(D(M)) \mbox{ such that }  g|_{M_1}= f \Big\}\end{equation}
  is open in $C(M)$.  In fact, since ${\mathcal U}' \subset {\mathcal V}'$,  the restriction $g|_{M_1} $ for each $g \in {\mathcal U}'$, belongs to $ {\mathcal V} $. Besides, the restriction that sends each map $g \in {\mathcal V'} $ to the map $f := g|_{M_1} \in {\mathcal V} $ is an open operator. Therefore, the  image ${\mathcal U}$ by that restriction, of the open family ${\mathcal U}' $, is open.

    Since   Theorem  \ref{theoremAA} holds for typical maps $g \in C(D(M))$, and ${\mathcal V'}$ is nonempty and open in $C(D(M)) $, it also holds for typical maps $g \in {\mathcal V}' $. Namely, there exists a sequence  $\{{\mathcal U}'_n\}_{n \geq 1}$ of open and dense subsets of ${\mathcal V}'$ such that assertions a) and  b)  hold for all $g \in \bigcap_ {n \geq 1} {\mathcal U}'_n \subset {\mathcal V}' $. Consider  the preimages ${\mathcal U}_n \subset {\mathcal V}$ defined in (\ref{eqnU}), of the sets ${\mathcal U}_n$.  Since ${\mathcal U}_n$ is open and dense in ${\mathcal V'}$, the above assertions imply  that ${\mathcal U}_n$ is open and dense in ${\mathcal V}$; hence also in $\C$.

If Theorem \ref{theoremAA}  holds for a map $g \in {\mathcal V}' \subset C(D(M))$, then  it also holds for the restriction $g|_{M_1} \in {\mathcal V} \subset \C$. We conclude that it holds for the countable intersection $\bigcap_{n \geq 1} {\mathcal U}_n$ of open and dense sets in $\C$. In other words, Theorem \ref{theoremAA} holds for typical maps in $\C$.
\end{proof}

 {The following theorem follows easily  from results of Hurley and his coauthors on the existence of
many periodic points for homeomorphism of manifolds of dimension at least two without boundary \cite{AHK,H1}.
We give a different proof which also works for manifolds with boundary and also in dimension one.}

\begin{theorem}
\label{TheoremAddedNonUniqErgodic}
Typical maps in $\C$ (resp.\ $\HM$) are not uniquely ergodic.
\end{theorem}

\begin{proof}
First, let us argue in $\C$. Recall that Lebesgue a.e.\ $x \in M$ belongs to the set $AA_1$ defined by Equality \eqref{eqnEquationAA}. We claim that
if $f \in \C $ is typical,   then
 for any $\e > 0$ and any  $x_0 \in AA_1$ the \nolinebreak set \linebreak $\{x \in AA_1: \mu_x \neq \mu_{x_0} \mbox{ and } d(\mu_x, \mu_{x_0}) < \e\}$ has positive Lebesgue measure.

To prove the claim we argue by  contradiction, assume that there exists $x_0 \in AA_1$ and $\e >0$ such that \begin{equation}
\label{eqn08}
\mbox{Leb}\{x \in AA_1\colon \mu_x \neq \mu_{x_0} \mbox{ and } d(\mu_{x_0}, \mu_x) < \e\} = 0.\end{equation}
Since $\mu_{x_0}$ is pseudo-physical,  by definition we have
$$\mbox{Leb}(\{x \in M\colon d(p\omega (x), \mu_{x_0}) < \e\}) >0.$$
Since $Leb(AA_1) = 1$ we deduce that
$$\mbox{Leb}(\{x \in AA_1\colon d(\mu_x, \mu_{x_0}) < \e\}) >0.$$
Combining this last assertion with Inequality (\ref{eqn08}), we obtain that
$$\mbox{Leb}(\{x \in AA_1\colon  \mu_x = \mu_{x_0}) \}) >0.$$
So, $\mu_{x_0}$ is a  physical measure, contradicting  Theorem \ref{theoremAA}, the claim is proven.

Since the measures $\mu_x $ are $f$-invariant for all $x \in AA_1$, we deduce  from the claim that $f$ is not uniquely ergodic.

Now let us prove Theorem \ref{TheoremAddedNonUniqErgodic} for typical $f \in \HM$.
From Theorem \ref{theoremLebesgue-a-e-isInShrinkingInterval}, there exists a periodic shrinking set $I$ {with arbitrarily small diameter}. Denote by $p$ the period of $I$, and by $x_0 \in I$ a periodic point of period $p$.
 Recall that $\mbox{diam}(f^j(I)) < \mbox{diam}(I)$.
Besides the sets $f^i(\bar I)$ are pairwise disjoint for $j \in \{0,\dots,p-1\}$, $f^p (\bar I) \subset I$,
and this inclusion is strict since $\mbox{diam}(f^p(I)) < \mbox{diam}(I)$. Therefore,    if $\mbox{diam}(I)$ is small enough, then
\begin{equation}\nonumber
\label{eqnKneqM} K:= \bigcup_{j=1}^{p} f^j(\overline I) \neq M.
\end{equation}

Denote by $\nu$ the atomic invariant measure supported on the periodic orbit of $x_0$.  We clain that
there exists $\delta>0$  such that for any Borel probability measure $\mu$ (not necessarily invariant), if $\mbox{dist}(\mu, \nu)< \delta$, then $\mu(I) >0$.

To prove the claim, we argue by contradiction. Assume that there exists $\mu_n \rightarrow \nu$ such that $\mu_n(I)=0$. Consider $\psi:M \mapsto [0,1]$ defined by
\begin{equation}
\label{eqn-psi}
\psi(x):=\frac{\mbox{dist}(x, M \setminus I)}{\mbox{dist}(x , f^p(\overline I)) + \mbox{dist}(x, M \setminus I) }.\end{equation}
This function is continuous and satisfies $\psi|_{f^p(\overline I)} =1, 0 < \psi(x) <1$ if $x \in I \setminus f^p(\overline I))$ and $\psi(x)= 0$ if $x \not \in I$. Since $\mu_n(I)=0$, we obtain $\int \psi \, d \mu_n =0$, and taking the weak$^*$-limit of $\mu_n$, we deduce $\int \psi \, d \nu=0$; hence $\psi=0$ $\nu$-a.e., which contradicts that $\nu(I) = 1/p >0$ and $\psi>0$ for all $x \in I$, finishing the proof of the claim.

Now, let us prove that $f$ is not uniquely ergodic. Arguing by contradiction, if $\nu$ were the unique invariant measure, then  $(1/n) \sum_{j=0}^{n-1} \delta_{f^j(x)}$ would converge uniformly to $\nu$ for all $x \in M$ as $n \rightarrow + \infty$. Therefore, there would exist $n_0 \geq 1$ such that
$$\mbox{dist} \Big(\nu, (1/n_0) \sum_{j=0}^{n_0-1} \delta_{f^j(x)}        \Big)< \delta \ \ \forall \ x \in M.$$
Applying  the claim for all $x \in M$, there would exist $0 \leq j = j(x) \leq n_0-1$ such that $f^j(x) \in I$. Since $I$ is shrinking periodic, we would deduce that $f^{n_0}(x) \in K$  for all $x \in M$. Therefore, applying assertion (\ref{eqnKneqM}), the image of $M$ by the homeomorphism $f^{n_0}: M \mapsto M$ would not be $M$, which is a contradiction.
\end{proof}

Now, we state the main lemma to be used to prove Theorem \ref{Theorem-main2}.

\begin{lemma} \label{Lemma-ConvexCombination}
Suppose that $f$ is a typical map in $\C$ (resp.\ $\HM$),
 $\mu_1$ is   $f$-invariant and supported on  $K = \bigcup_{j= 0}^{p-1} f^j(\I)  $, where $  I$ is a periodic shrinking set of period $p $, and that  $\mu_2$ is   $f$-invariant such that $\mu_2(K)=0$. Then, no convex combination $ \nu= \lambda \mu_1 + (1 - \lambda) \mu_2$  with $0 < \lambda < 1$  is pseudo-physical.
\end{lemma}

\begin{proof} Since $I$ is a shrinking periodic set with period $p$,  it is is open, $\overline{I}$ is an $m$-simplex, and $f^p(\overline{I}) \subset I$. Consider the continuous real function $\psi:M \rightarrow [0,1]$ defined by Equality (\ref{eqn-psi}).
It satisfies
 $\psi|_{f^p(\overline I)}= 1$, $0 < \psi(x) <1$ if $x  \in I \setminus f^p(\overline I) $, and $\psi(x) = 0$ if $x \not \in I$.

Since $\mu_1 $ is supported on $K =   \bigcup_{j= 0}^{p-1} f^j(\overline I) $,  we can apply Lemma \ref{lemmaInvMeasOnShrinkingIntervals} to deduce that $$\mu_1(\overline I) = \mu_1 (f^p(\overline I)) = \frac{1}{p}; \ \ \mbox{ hence }  \mu_1 (I \setminus f^p(\overline I)) = 0.  $$ We obtain $
  \int \psi \, d \mu_1 = \mu_1 (f^p(\overline I)) =  {1}/{p} $ and $ \int \psi \, d \mu_2 = \mu_2 (I) = 0.$
  Therefore
 $ 0 <\int \psi \, d \nu =  \lambda /p  <  {1}/{p}.  $

Choose $\e >0$  such that for any measure $\mu$, if $d(\nu, \mu) < \e$, then $\int \psi \, d \mu > 0$ and $\int \psi \, d \mu <   1 /p. $

Consider the set $$A'_{\e}(\nu) \defeq    \{ x \in AA: \mbox{ dist} (\mu_x, \nu) < \e\},$$
where the set $AA$ is defined in Equality (\ref{eqnEquationAA}).
We claim that  the set $A'_{\e} (\nu)$ is empty. Arguing by contradiction, assume that there exists $x \in A'_{\e}(\nu)$.
Then,  from the choice of $\e$, we have   \begin{equation}
   \label{eqn07}0 <\int \psi \, d \mu_x < 1/p, \end{equation} and  from (\ref{eqn-mu_x}) we deduce that there exists $n_0 \geq 1$ such that $$ \frac{1}{n} \sum_{j= 0} ^{n-1} \psi (f^j(x)) = \int \psi \, \Big (\frac{1}{n} \sum_{j= 0}^{n-1} \delta_{f^j(x)}
   \Big )  >0 \ \forall \ n \geq n_0.$$
Thus,  there exists $n_1 \geq 1$ such that $\psi(f^{n_1}(x))    >0$; hence $f^{n_1}(x) \in I$. This implies that  the future orbit of $f^{n_1}{x}$ is contained in $K$. Hence,  $\mu_x$ is supported on $K$. Since $\mu_x$ is $f$-invariant, we apply Lemma \ref{lemmaInvMeasOnShrinkingIntervals} to deduce that $\mu_x (f^p(\overline I)) = 1/p$ and $\mu_x(I \setminus f^p(\overline I)) = 0$. Therefore,
$$\int \psi \, d \mu_x = \mu_x(f^p(\overline I)) = \frac{1}{p},$$
contradicting  \eqref{eqn07}. We have proved that $A'_{\e} (\nu) $ is   empty.

The definition of  pseudo-physical measure $\mu$   asserts that  $$\mbox{Leb}(A_{\e}(\mu)) >0 \ \ \forall \ \e >0. $$
But \eqref{eqAA} implies that $Leb(A'_\e(\nu)) = Leb(A_\e(\nu))$, which  equals zero because
 $A'_{\e}(\nu) = \emptyset$.  We conclude that $\nu$ is not  pseudo-physical.
\end{proof}

\subsection{End of the proof of Theorem \ref{Theorem-main2}}

{
\begin{proof}
From \cite[Theorem 1.3]{CE1} the set $\O_f$ is closed. Let us prove that  its interior  in ${\mathcal M}_f$ is empty.
Fix $\mu \in \O_f$.  Since a typical map is not uniquely ergodic (recall Theorem \ref{TheoremAddedNonUniqErgodic}), there exists an ergodic measure $\nu \neq \mu$. Denote by $\e:= d(\mu_x, \nu)>0$.
Consider an arbitrary $\delta \in (0,\e/2)$.

From Proposition \ref{propositionAA_1}, we can find $x_0 \in AA_1$ such that $d (\mu_{x_0},\mu) < \delta$. So, to prove that $\mu$  does not belong to the interior of ${\mathcal O}_f$ in ${\mathcal M}_f$, it is enough to prove that the $\delta$-neighborhood of $\mu_{x_0}$ in ${\mathcal M}_f$ is not contained in ${\mathcal O}_f$. To do that, it is enough to find a sequence $\{\mu_n\}_n \subset {\mathcal M}_f$ converging to $\mu_{x_0}$ and such that $\mu_n \not \in {\mathcal O}_f$.

Applying   Theorem \ref{theoremLebesgue-a-e-isInShrinkingInterval} and Equality (\ref{eqAA1}), and taking into account the second part of Proposition \ref{propositionAA_1},   the point $x_0$ could  be chosen to belong  to arbitrarily small eventually periodic shrinking sets. Therefore, for any fixed $q \in \mathbb{N}^+$ (that will be chosen later),  the probability measure $\mu_{x_0}$ is supported on the orbit of a periodic shrinking set  $I_q$ of  diameter smaller than $1/q$. Denote by $p$ the period of $I_q$, and denote
$$K = \bigcup_{j=0}^{p-1} f^j(\overline I_q).$$
We have $\mu_{x_0}(K)=1$. So,
applying Lemma \ref{lemmaInvMeasOnShrinkingIntervals} we deduce $$\mu_{x_0} (f^j(\overline I))= {1}/{p} \ \ \ 0 \leq j \leq p-1.$$

Applying Lemma \ref{lemma-dist}, construct $q \geq 1$ such that if $\mu$ is a probability measure satisfying $\mu(f^j(\overline I)) = \mu_x(f^j(\overline I)) = 1/p$ with  $\mbox{diam}(f^j(\overline I)) <1/q$ for all $0 \leq j \leq p-1$, then $d(\mu, \mu_x)< \delta$. Since $d(\mu, \nu)= \e > \delta$, we deduce that $\nu(f^j(\overline I)) \neq 1/p$ for some $j$. But applying Lemma \ref{lemmaInvMeasOnShrinkingIntervals} $\nu(f^j(\overline I)) = \nu(K)/p$ for all $j$, with $\nu(K) \in \{0,1\}$. Since $\nu$ is ergodic, we conclude that $\nu(K)=0$. Therefore, applying Lemma \ref{Lemma-ConvexCombination}, $\mu_n:=\lambda_n \mu_x + (1- \lambda_n) \nu \in {\mathcal M}_f \setminus {\mathcal O}_f$ for all $0 <\lambda_n<1$.
Taking $\lambda_n \rightarrow 1^-$, we obtain $\mu_n \rightarrow \mu_x$, with $\mu_n \in {\mathcal M}_f \setminus {\mathcal O}_f$, as wanted.
\end{proof}
}

\section{Pseudo-physical,  ergodic and periodic measures.}\label{sec5}
A \em  $\delta$-pseudo-orbit of $f$ \em is a sequence $\{y_n\}_{n \in  \mathbb{N}} \subset M$ such that $$dist(f(y_n), y_{n+1} ) < \delta \ \ \forall \ n \in \mathbb{N}.$$
A \em $\delta$-pseudo-orbit  $\{y_n\}_{n \in  \mathbb{N}} $ is  periodic \em with period $p \geq 1$, if $$y_{n+p}= y_n \ \ \ \forall \ n \in \mathbb{N}.$$
A  map $f \in \C$ has the   \emph{periodic shadowing property} if for all $\e > 0$, there exists $\delta>0$ such that,  if $\{y_n\}_{n \in \mathbb{N}}$ is any periodic $\delta$-pseudo-orbit, then, at least one  periodic orbit $\{f^n(x)\}_{n \in \mathbb{N}}$ satisfies
$$dist(f^n(x), y_n) < \e \ \ \ \forall \ n \in \mathbb{N}.$$

We will also use the following result, which is the accumulation of many authors' work.
\begin{theorem}
\label{TheoremOprocha&als}
Typical maps $f \in \C$ (resp.\  $\HM$) have the periodic shadowing property.
\end{theorem}
\begin{proof} The statement follows from \cite[Theorem 1.2]{Oprocha&als} and \cite[Theorem 1]{PP}.  Earlier special cases where treated in \cite{Y} and \cite{O}.
\end{proof}

Recall that ${\mathcal E}_f$ denotes the set of ergodic measures, and $\mbox{Per}_{f}$ denotes the set of invariant measures supported on periodic orbits of $f$.
As a consequence of Theorem \ref{TheoremOprocha&als}:
\begin{corollary}
\label{corollaryOprocha&als}
For a typical map $f \in \C$ (resp.\  $f \in \HM$)  we have $\mathcal E_f \subset \overline{\mbox{Per}_{f}}. $
\end{corollary}

\noindent{\bf Remark. }
 The closing lemma and the closeability properties (see for instance \cite[Definition 2.1]{Coudene-S} and \cite[Definitions 4.1 and 4.5]{Gelfert-K}), state  that for all $\epsilon>0$ there exists $\delta>0$ such that, for a finite piece  of orbit $\{f^j(y\}_{0 \leq j \leq p}$ verifying $dist(f^p(y), y) < \delta$,  there exists a periodic point $x$, whose period  is $p$  or at least   $(\epsilon \cdot p)$- near  $p$, satisfying $dist(f^j(y), f^j(x)) < \epsilon$ for all $0 \leq j \leq p$.  It is  known that the closing lemma and the closeability properties imply   ${\mathcal E}_f \subset \overline{\mbox{Per}_f}$ (see for instance  \cite[Lemma 2.2]{Coudene-S} and  \cite[Theorems 4.10]{Gelfert-K}).
 Nevertheless,  we do not know if typical maps in $C^0(M)$  satisfy the closing lemma or the closeability properties.  

\begin{proof}[Proof of Corollary \ref{corollaryOprocha&als}]
From the definition of distance in the weak$^*$ topology of the space of probability measures it is standard to check that   for all $\e_0 >0$, there exists $\e>0$, such that, for any two points $x_1, x_2 \in M$, $$dist(x_1, x_2) < \e \ \ \Rightarrow \ \ d(\delta_{x_1}, \delta_{x_2})< \e_0.  $$

Fix any $\mu  \in {\mathcal E}_f$. Since $\mu$ is  ergodic, we have
$p \omega (x) = \{\mu\}  \ \ \mbox{ for } \mu\mbox{-a.e. } x \in M.$ Fix such a point $x$;  then there exists $n_0 \geq 1$ such that
\begin{equation}
\label{eqn11}
d \left (\frac{1}{n} \sum_{j= 0}^{n-1} \delta_{f^j(x)}, \ \mu \right ) < \e_0 \ \ \ \forall \ n \geq n_0.\end{equation}
Given $\e$, choose $\delta >0$ given by Theorem \ref{TheoremOprocha&als}. By Poincar\'{e} Recurrence Lemma, the point $x$ can be chosen to be recurrent. Thus
$$ d(f^{p - 1} (x), x) < \delta \ \ \mbox{ for some } p \geq n_0.$$
Construct the periodic $\delta$-pseudo-orbit  $\{y_n\}_{n \in \mathbb{N}}$ of period $p$ defined by   $y_n= f^n(x)$ for all $0 \leq n < p$,
$y_{n + p} = y_n$ for all $n \geq 0$.  Applying Theorem \ref{TheoremOprocha&als}, there exists a periodic orbit $\{f^n(z)\}_{n \geq 0}$, such that
$dist (f^n(z), y_n) < \e \ \ \forall \ n \geq 0$.
By construction,  if  $ip \le  n < (i+1)p$    and   $i \ge 0$ then
$dist(f^n(z), f^{n-ip}(x)) < \e $.
Thus, from the choice of $\e$, we obtain
$d(\delta_{f^n(z)}, \delta_{f^{n-ip}(x)}) < \e_0$
Denote by $q$ the period of $z$.
Taking into account that  balls are convex in the
 weak$^*$-distance in the space of probabilities, we deduce
$$d\left ( \frac{1}{qp} \sum_{j= 0 }^{qp-1} \delta_{f^j(z)},  \ \ \frac{1}{qp}  q \cdot \sum_{j= 0 }^{p-1} \delta_{f^j(x)}   \right) < \e_0. $$
For the atomic invariant measure $\nu$ supported on the periodic orbit of $z$, we have
$$\nu = \frac{1}{q} \sum_{j= 0}^{q-1} \delta_{f^j(z)} =  \frac{1}{qp} \sum_{j= 0 }^{qp-1} \delta_{f^j(z)}.$$
Thus,
$$d \left ( \nu ,  \ \ \frac{1}{p} \sum_{j= 0 }^{p-1} \delta_{f^j(x)}   \right) < \e_0.$$
Together with  (\ref{eqn11}), this implies that for any $\e_0$
the given ergodic measure $\mu$ is $2  \e_0$-approximated by  some measure $\nu \in \mbox{Per}_{f}$,
finishing the proof of   Corollary \ref{corollaryOprocha&als}.
\end{proof}

 An invariant measure $\mu$ is called \em infinitely shrinked \em if there exists a sequence $\{I_q\}_{q \geq 0}$ of periodic shrinking intervals $I_q$, of periods $p_q$, such that $\mbox{diam}(I_q) < 1/q$ and $\mu(\bigcup_{i= 1}^{p_q}f^j(\I_q) )= 1$ for all $q \geq 1$.
We denote by $\mbox{Shr}_f \subset {\mathcal M}_f$ the set of  infinitely shrinked  invariant measures.
Define
$$AA_2  \defeq \{ x \in AA_1: \ \  \mu_x \in \mbox{Shr}_f\}. $$
\begin{theorem}
\label{theoremSRB-l=ClosureSrh}
For a typical map $f \in \C$ (resp.\ $f \in \HM$),  $$ \mbox{ Leb}(AA_2)  = 1 \ \ \mbox{ and  } \ \  \O_f  =\overline {\{\mu_x \colon  x \in AA_2\}} =  \overline {\mbox{\em Shr}_f}.$$
\end{theorem}

\begin{proof}
From Theorem \ref{theoremLebesgue-a-e-isInShrinkingInterval}, Lebesgue-a.e.~$x \in M$ belongs to a sequence of eventually periodic or periodic shrinking sets $J_q$ with $\mbox{diam}(J_q) < 1/q$. Every eventually periodic shrinking set $J_q$  wanders under  $f$ until it drops into a periodic shrinking set $I_q$ with
$ \mbox{diam}(I_q)< \mbox{diam}(J_q)< 1/q$.  By the definition of periodic shrinking set, every point of $I_q$ has all the measures of $p\omega (x) $ supported on the compact set $$K_{x,q} \defeq \bigcup_{j= 0}^{p_q-1} f^j(\overline I_q).$$ In particular for Lebesgue almost all $x \in AA_1$, the limit measure $\mu_x$  defined by   (\ref{eqn-mu_x}), is supported on $K_{x,q}$. Thus, for a.e.\ $x \in AA_1$ we have $\mu_x \in \mbox{Shr}_f$. Taking into account (\ref{eqAA1}), the above assertion implies  that the set $AA_2$ has full Lebesgue measure.

By construction,  $AA_2 \subset AA_1$. So, applying Proposition \ref{propositionAA_1}, we obtain:
$$ \overline {\{\mu_x \colon  x \in AA_2\}} \subset \overline {\{\mu_x \colon  x \in AA_1\}} = \O_f. $$
To obtain the opposite inclusion, we apply \cite[Theorem 1.5]{CE1}):   $\O_f$ is the minimal weak$^*$-compact set of probability measures, that contains $p \omega (x)$ for Lebesgue a.e.~$x$. Since $ \overline {\{\mu_x \colon  x \in AA_2\}}$ is weak$^*$-compact and contains $p\omega_x = \{\mu_x\}$ for Lebesgue almost all $x$ (because $\mbox{Leb}(AA_2) = Leb (M)$), we conclude that
    $$   \overline {\{\mu_x \colon  x \in AA_2\}} \supset \O_f.$$

    The inclusion   $\overline {\{\mu_x \colon  x \in AA_2\}} \subset \overline{\mbox{Shr}_f}$ follows trivially from the definition of the set $AA_2$. Now, let us prove the opposite inclusion. We will prove that every shrinking measure is pseudo-physical.  Let $\mu \in \mbox{Shr}_f$. For any  $\e>0$, choose  $q \geq 1$ as in Lemma \ref{lemma-dist}.  By the definition of shrinking measure, there exists a periodic shrinking set $I'_{q'}$ such that $\mu$ is supported on
    $$K_\mu  = \bigcup_{j= 0}^{p_{q'}-1} f^j(\overline{ I'_{q'}}) $$
 for the periodic shrinking set $I'_{q'}$ of period $p_{q'}$, such that  $\mbox{diam}(\overline {I'_{q'}}) <1/{q'}$; hence  $\mbox{diam}(f^j(\overline I'_{q'})) <1/{q'} \ \ \forall \ 1 \leq j \leq p_{q'}$,
$\mu(K_{\mu}) = 1$.

    Besides, for any point $x \in I'_{q'}$, any  measure in $p\omega_x$ is also supported on $K_{\mu}$. If additionally $x \in AA_1$, then $p\omega_x = \{\mu_x\}$, so  $\mu_x (K_{\mu}) = 1$.
    Finally, applying Lemmas \ref{lemma-dist} and  \ref{lemmaInvMeasOnShrinkingIntervals}, we deduce that
      the measures $\mu \in \mbox{Shr}_f$ and $\mu_x$ given above satisfy $$d(\mu, \mu_x) < \e \ \mbox{ for any } x \in I_q \cap AA_1. $$ Since $Leb (I'_{q'} \cap AA_1) = Leb (I'_{q'}) >0$, the basin $A_\e (\mu)$ has positive Lebesgue measure; namely $\mu $ is pseudo-physical.

    We have shown that every shrinking measure is pseudo-physical. Since the set $\O_f$ of pseudo-physical measures is closed, we conclude $$\overline {\mbox{Shr}_f} \subset \mathcal O_f, $$
finishing the proof of Theorem \ref{theoremSRB-l=ClosureSrh}.
\end{proof}

\begin{theorem} \label{theoremShrIncludedErgodic-SRB-l-includedClosureErgodic}
For any map $f \in \C$ (resp.\  $f \in \HM$), if $\mu \in \mbox{Shr}_f$, then it is ergodic.
\end{theorem}
Before proving Theorem \ref{theoremShrIncludedErgodic-SRB-l-includedClosureErgodic} let us deduce its main consequence:
\begin{corollary} \label{corollaryShrink}
For a typical map $f \in \C$ (resp.\ $f \in \HM$):
$$\O_f  = \overline{\mbox{Shr}_f}\subset \overline{{\mathcal E}_f} =  \overline{\mbox{Per}}_f.$$
\end{corollary}

The corollary immediately follows by combining Corollary \ref{corollaryOprocha&als} with
Theorems  \ref{theoremSRB-l=ClosureSrh} and \ref{theoremShrIncludedErgodic-SRB-l-includedClosureErgodic}.
At the end of the next section we will prove that  for typical maps these sets are all equal.

\begin{proof}[Proof of Theorem \ref{theoremShrIncludedErgodic-SRB-l-includedClosureErgodic}]
Fix $f \in \C$. Suppose $\mu \in \mbox{Shr}_f$, and  $\mu_1, \mu_2 \in \mathcal{M}_f$ such that
\begin{equation}
\label{eqn13}
\mu = \lambda \mu_1 + (1 -\lambda)\mu_2, \ \ \ \mbox{with} \ \  0 < \lambda <1,\end{equation}
We shall prove that $\mu_1 = \mu_2 = \mu$; namely  $\mu$ is extremal in the convex compact set of invariant measures; hence ergodic.

Take arbitrary $\e >0$ and fix $q \geq 1$ as in Lemma \ref{lemma-dist}. By the definition of infinitely shrinking measures, there exists a periodic shrinking set $I_q$, with $\mbox{diam}(I_q) < 1/q$, and period $p_q$, whose $f$-orbit $K_q$ supports $\mu$.  The definition of periodic shrinking set and Lemma \ref{lemmaInvMeasOnShrinkingIntervals} tell us:
$$\mu(f^j(\overline I_q))  = \frac{1}{p_q},  \ \ \ \mbox{diam} (f^j(\overline I_q)) < 1/q \ \ \ \ \forall \ 1 \leq j \leq p_q.$$

Since $\mu(K_q) = 1$,  from  (\ref{eqn13}) we deduce $\mu_1(K_q) = \mu_2(K_q) = 1 $. Applying  Lemma \ref{lemmaInvMeasOnShrinkingIntervals} we obtain
 $$\mu_1(f^j(\overline I_q))  = \mu_2(f^j(\overline I_q))  = \frac{1}{p_q}\ \ \ \ \forall \ 1 \leq j \leq p_q. $$
So,   Lemma \ref{lemma-dist} implies
$d(\mu_1, \mu) < \e,$ and $d(\mu_2, \mu)  < \e$.
Since $\e> 0$ is arbitrary, we conclude that $\mu = \mu_1 = \mu_2$; hence $\mu $ is ergodic.
\end{proof}

\section{All ergodic measures are pseudo-physical.}\label{sec6}
\begin{definition} \em  \label{definition-qShrinking}
Let $q \geq 1$ and $x_0$ be a periodic point with period $r \geq 1$. We call the  (invariant) measure  $$\nu= \frac{1}{r} \sum_{j=0} ^{r-1}\delta_{f^j (x_0)} \in Per_f$$   a \em  $q$-shrinked periodic measure,  \em  if there exists some periodic shrinking set  $I $, with diameter smaller than $1/q$, with period $p \geq 1$ such that $\nu$ is supported on   $K \defeq \bigcup_{j= 1}^p f(\I)$. From the definition of periodic shrinking set,    the period $p$ must divide $r$.
We denote by $\mbox{Shr}_q\mbox{Per}_{f}$ the set of $q$-shrinked periodic measures.

We say that an invariant measure $\mu$,  \em is $\e$-approached by $q$-shrinked  periodic measures \em if there exists $\nu \in \mbox{Shr}_q\mbox{Per}_{f}$ such that $d(\mu, \nu) < \e $. We denote by $\mbox{AShr}_{\e, q}\mbox{Per}_{f}$ the set of measures that are  $\e$-approached by $q$-shrinked periodic measures.

\end{definition}

\begin{theorem} \label{theoremBigcapClosShr_q&SRBl-l}
For any map $f \in \C$ (resp.\ $f \in \HM$)
$$  \bigcap_{\e >0}\bigcap_{q \geq1}\overline{\mbox{AShr}_{\e, q}\mbox{Per}}_f  \subset \O_f.$$

\end{theorem}
\begin{proof}
Fix  $\e > 0$, and choose $q \geq 1$ as in Lemma \ref{lemma-dist},  such that $1/q < \e$.
For any $\mu_q  \in \mbox{AShr}_{1/q, q}\mbox{Per}_f $, denote by $\nu_q $ a measure in $ {\mbox{Shr}_q\mbox{Per}}_f$  such that $$d(\mu_q, \nu_q) < 1/q< \e.$$
Consider  the periodic shrinking set $I = I(\nu_q)$ and the compact set $K$ for $\nu_q$ from   Definition \ref{definition-qShrinking}. From the definition of periodic shrinking set,  any point $x \in I$ satisfies $f^n (x) \subset K= \bigcup_{j= 0}^{p-1} f^j(\I)$ for all $n \geq 0$. So, any  measure $\mu_x \in p\omega (x)$ is supported on $K$. Also $\nu_q$ is supported on $K$. Thus, applying Lemma \ref{lemmaInvMeasOnShrinkingIntervals}, we deduce that
$$\mu_x( f^j(\I)) = \nu_q(f^j(\I)) = \frac{1}{p} \ \ \forall \ 1 \leq j \leq p; \ \ \ \mbox{diam}(f^j(\I)) < \frac{1}{q}.$$
Now, from Lemma \ref{lemma-dist}, we obtain  $d(\nu_q, \mu_x) < \e$ for all $\mu_x \in p\omega(x)$, for all $x \in I$. Thus, for any $x \in I$
\begin{equation}
\label{eqn14}
d(p \omega(x), \nu_q)< \e,  \mbox{ hence } d(p \omega(x), \mu_q)< 2 \e.\end{equation}

Note that when we vary the value of $\e >0$, the value of $q$, and thus also the measures $\nu_q$ and $\mu_q$ and the set $I$, may change. So, from the above inequality we \em can not deduce \em that each $\mu_q$ is pseudo-physical. Nevertheless, we have proved that for any fixed value of  $\e> 0$ there exists $q \geq 1$ such that Inequality (\ref{eqn14}) holds  for all $\mu_q  \in \mbox{AShr}_{1/q, q}\mbox{Per}_f $.

Now, consider any measure $\mu'_q \in  \overline{\mbox{AShr}_{1/q, q}\mbox{Per}_f}$. Thus, there exists $\mu_q \in \mbox{AShr}_{1/q, q}\mbox{Per}_f$ such that
$$d(\mu'_q, \mu_q) < \e.$$
Combining this with (\ref{eqn14}) we deduce that,  for  all $\e >0$   there exists   $q \geq 1$ such that, for any measure $\mu'_q  \in \overline{\mbox{AShr}_{1/q, q}\mbox{Per}_f}$ there exists an open set $I$ (the periodic shrinking set $I(\nu_q)$ for the measure $\nu_q$ associated to $\mu'_q$)  such that
\begin{equation}
\label{eqn15}
d(p \omega(x), \mu'_q) < 3\e \ \ \ \forall \ x \in I.\end{equation}

So, if $\mu \in \bigcap_{\e >0}  \bigcap_{q \geq 1}\overline{\mbox{AShr}_{\e, q}\mbox{Per}_f}$, then,  for all $\e >0$ there exists an open set $I$ satisfying assertion (\ref{eqn15}). Thus  $\mbox{Leb}(A_{3\e}(\mu)) \ge \mbox{Leb}(I) >0$ for all $\e >0$; hence $\mu \in \O_f$, as wanted.
\end{proof}

The last ingredient of the proof of Theorem \ref{Theorem-main} is the following theorem.

\begin{theorem}
\label{TheoremSRB-l=ClosureErgodic}
For a typical map $f \in \C$ (resp.\ $f \in \HM$),
\begin{equation}
\label{eqn16}
\mbox{Per}_f \subset \bigcap_{q \geq1}\overline{\mbox{AShr}_{\e, q}\mbox{Per}_f} \ \ \forall \ \e>0.\end{equation}
\end{theorem}

Before proving Theorem \ref{TheoremSRB-l=ClosureErgodic}, let us introduce the following definition:

\begin{definition}
\label{definitionP_q,r} \em

Fix  $q,r \in \N^+$.
  A \em good $q, r$-covering  \em $\mathcal U_{q,r}$ for $f \in \C$ (resp.\ $\HM$), is a finite family of \em open \em simplexes (i.e., the interiors of simplexes)  such that

(1)  $\mathcal U_{q,r}$ covers the compact set  $Per(f,r) \defeq \{x \in M: f^rx = x\}$.

(2)  $\mbox{diam}(U_i)< 1/q$ for any $U_i \in \mathcal U_{q,r}$.

(3) For any $U_i \in \mathcal U_{q,r}$, there exists a periodic shrinking set $I_i$, with period $p_i \leq r$, with $p_i$ that divides $r$, such that $\I_i \subset U_i$.

\vspace{.2cm}

We call a map $f \in \C$ (resp.\ $\HM$)  \em a good $q,r$-covered map,  \em  if  there exists a  good $q, r$-covering  $\mathcal U_{q,r}$ for $f$. We denote by ${\mathcal P}_{q,r} \subset \C$ (resp.\ $\HM$) the set of all good $q,r$-covered maps.
\end{definition}

\begin{proof}[Proof of Theorem \ref{TheoremSRB-l=ClosureErgodic}] We claim that, for fixed $q,r \geq 1$, the set ${\mathcal P}_{q,r}$  is open in $\C$.
Fix $f \in {\mathcal P}_{q,r} $, and denote its good $q,r$-covering by  $\mathcal U_{q,r} = \{U_1, U_2, \ldots, U_h\}$. The compact set $K= M \setminus \bigcup_{i= 1}^h U_i$ does not intersect the compact set $\{f^r(x) =x\}$. Let us prove that for all $g \in \C$ (resp.\ $\HM$) close enough to $f$, the same compact set $K$ (defined for the \em same \em covering ${\mathcal U}_{q,r}$) does not intersect   $\{g^r(x) = x\}$.  In fact,    the real function   $\phi_f(\cdot) : = \mbox {dist} (f^r(\cdot), \cdot) $  depends continuously on $f$. Since $\min _{x \in K} \phi_f(x) >0$, we deduce $\min _{x \in K} \phi_g(x) >0$ for all $g \in \C$ (resp.\ $\HM$) close enough to $f$.
In other words, ${\mathcal U}_{q,r}$ also covers the fixed points of $g^r$.  Thus, the good $q,r$-covering of $f$, is also a covering satisfying conditions (1) and (2) of Definition \ref{definitionP_q,r},  for any $g \in \C$ (resp.\ $\HM$) close enough to  $f$.
Now, let  us prove that Condition (3) for $g$  is satisfied   by the same   covering ${\mathcal U}_{q,r}$, provided that $g $ is close  enough to
 $f$.
Consider a $f$-shrinking periodic set $I_i \subset U_i \in  {\mathcal U}_{q,r}$, of period $p_i$.
Now  $I_i$ is a periodic shrinking set with the same period $p_i$  for all $g$ sufficient close  to $f$. Since the family $\{I_i\}_{1 \leq i \leq h}$ of shrinking periodic sets to be preserved is finite, we conclude that
(3) is also satisfied for any $g$ sufficiently close to $f$
and thus ${\mathcal P}_{q,r}$ is open in $\C$ (resp.\ $\HM$).

In Lemma \ref{lemmaP_qrIsDense}, we will  prove that  ${\mathcal P}_{q,r}$ is  dense in $\C$ (resp.\ $\HM$). Let us conclude the proof of Theorem \ref{TheoremSRB-l=ClosureErgodic} assuming that Lemma \ref{lemmaP_qrIsDense} is proven.
Observe that, for fixed $q,r \geq 1$, any $f \in {\mathcal P}_{q,r}$ has the following property:
any   point $x_0$ fixed by $f^r$ (in particular any periodic point $x_0$ of period $r$) is $(1/q)$-near all the points of a periodic shrinking set $I_0$ with diameter smaller than $1/q$, and with period $p_0  \leq r$, $p_0 $ dividing $r$.

Besides, any periodic shrinked set of period $p_0$ has at least one periodic point $y_0$, fixed by $f^{p_0}$. We deduce that $I_0 $, whose diameter is smaller than $1/q$, contains  a periodic point $y_0$. Using the  definition of the set of measures $\mbox{Shr}_q \mbox{Per}_f$, we summerize this assertion as follows:
 \begin{equation}
\label{eqn18} \mbox{for all }  \ x_0 \mbox{ with period } r, \ \mbox{ there exists }  \end{equation} $$
\nu_0 \defeq  \frac{1}{r} \sum_{j= 0}^{r-1} \delta_{f^j(y_0)} \in \mbox{Shr}_q \mbox{Per}_f,
\mbox{ with } \ \  dist(y_0, x_0) < 1/q.$$

For   $r \geq 1$ fixed, consider $f \in \bigcap_{q \geq 1} \mathcal P_{q,r}$.  Let $\mu_0 \defeq  \frac{1}{r} \sum_{j= 0}^{r-1} \delta_{f^j(x_0)}$.
Fix $\e > 0$ and choose $q' \ge 1$ as in Lemma \ref{lemma-dist}.
Since $f^j$ is continuous for each $j$, we can find a $q > q'  $ so that if $dist(x,y) < 1/q$, then $dist(\delta_{f^j x}, \delta_{ f^j y}) < 1/q'$ for all $j \in  \{0,1,\dots, r\}$.
Then Lemma \ref{lemma-dist}  implies that $d(\mu_0,\nu_0) < \e$.

We have shown that for any given periodic orbit $\{f^j(x_0)\}_{ 0 \leq j \leq r-1}$ of period $r$, the distance between the periodic measure  supported on it, and some measure $\nu_q \in \mbox{Shr}_q \mbox{Per}_f$, for all $q$ large enough, is smaller than $\e$.
  In other words, any periodic measure supported on a periodic orbit of period $r$, belongs to $\bigcap_{q \geq 1} \overline {\mbox{AShr}_{\e, q}\mbox{Per}_f}$ for all $\e>0$.

  Finally, if $f \in {\mathcal P}\defeq \bigcap_{r \geq 1} \bigcap_{q \geq 1}\mathcal P_{q,r}$, then all its periodic measures (supported on periodic orbits of any period $r$) will belong to $\overline{\mbox{AShr}_{\e, q}\mbox{Per}_f} $ for all $q \geq 1$ and for all $\e>0$. In brief, if $f \in   \mathcal P$, then
  $$\bigcap_{q \geq 1} \mbox{Per}_f \subset \overline{\mbox{AShr}_{\e, q}\mbox{Per}_f}  \ \ \ \forall \ \e > 0. $$

   As ${\mathcal P}_{q,r}$ is open and dense in $\C$ (resp.\ $\HM$), the maps $f \in \mathcal P$ are typical. This ends the proof of Theorem \ref{TheoremSRB-l=ClosureErgodic}, provided that Lemma \ref{lemmaP_qrIsDense} is proven.
\end{proof}
\begin{lemma}
\label{lemmaP_qrIsDense}

For each  $q,r \geq 1$, the set ${\mathcal P}_{q,r}$
is dense in $\C$ (resp.\ $\HM$).
\end{lemma}

\begin{proof}   We  will use the notation that we defined in Subsection \ref{sec-setup}. Fix $f \in \C$ (resp.\ $\HM$) and $\e>0$.
Consider $0<\delta \defeq \frac{1}{q'}< \min\Big \{\frac{1}{q}, \ \frac{\e}{2}  \Big \}$  such that $dist(x,y) < \delta \ $ implies $  dist(f(x) , f(y))< \e$.
We wish to $\e$-perturb $f$  into new map $g \in {\mathcal P}_{q,r}$.
Consider a triangulation $\mathcal T := \{\overline T_1,\dots,\overline T_h \}$ of $M$ such that the diameters of all the simplexes $\overline T_i$ are at most $\delta/3$.
By modifying $\mathcal T$ we can suppose that if  $\T_i \cap Per(f,r) \ne \emptyset$ then  $T_i \cap Per(f,r) \ne \emptyset$,  where $T_i = \mbox{int}(\overline T_i)$; and furthermore, for some
$x_i  \in T_i \cap Per(f,r)$   its orbit $\{x_i, f(x_i), \dots, f^{r-1}(x_i) \} \cap \partial \mathcal T = \emptyset$.
 Indeed, suppose some $\overline T_i$ does not satisfy these conditions.
 If no point of $\{x : f^r x = x  \}$ is in the interior of $\overline T_i$ but some is in $\partial T_i$, then we can modify the triangulation by first slightly moving one corner of $\overline T_i$ so that in the resulting triangulation  some point $x_i \in Per(f,r)$ belongs to the interior $T_i$, and all the other simplexes of $\mathcal T$ which satisfied this condition before,  still satisfy it.

If the orbit of $x_i$ does not lay in the interior of the simplexes,   let $1 \le  j \le r-1$ be the first time that $f^j(x_i) \in \partial \mathcal T$.  Again
we modify the triangulation by slightly moving one corner of the triangulation  so that in the modified triangulation $f^j(x_i) \not \in \partial \mathcal T$,
and all the other simplexes of $\mathcal T$ which satisfied this condition before,  still satisfy it.
  Repeat this procedure a finite number of times to
produce the desired triangulation. We can assume that the perturbations are so small that in  the resulting triangulation the  diameters of all the simplexes are at most $\delta/2$.

Let $\mathcal T^{(1)} \defeq  \{\overline T_1^{(1)}, \dots, \overline T_{l}^{(1)}  \} \defeq \{\overline T_i \in  \mathcal T:  {T_i }\cap Per(f,r) \ne \emptyset \} $  and for each $i$ let $x_i$ be a point in $T_i^{(1)} \cap Per(f,r)$. Using $x_i$ as a centroid, consider a real number $\lambda_{\overline T_i^{(1)}} >1$ near enough 1, such that $\lambda_{\overline T_i^{(1)}}{\overline T_i^{(1)}}$ is well defined.
Let $\lambda_{\mathcal T} \defeq \min \{ \lambda_{T_i^{(1)}}: 1 \le i \le l\}$.
For $\lambda \in (1, \lambda_{\mathcal T})$ we have $x_i \in T_i \subset \T_i  \subset \lambda T_i$, and  $  \{\lambda T_1,\dots,\lambda T_l \}$ is an open cover of $Per(f,r)$.
We choose   such a value of  $\lambda_{max} \in (1,\lambda_{\mathcal T})$, so close to $1$ such that besides, each of the points $x_i$ does not lie inside $\la \T_j$ for any $j \ne i$, and furthermore such that the diameters of the $\lambda T_i$ are at most $\delta$.
For the rest of the proof we fix such a value of $\lambda \in (1,\lambda_{\mathcal T})$.

Denote    $$F : = \{x_{i}: 1 \leq i \leq l\}, \ \ \ \ \ F^r \defeq \bigcup_{j= 0}^{r-1} f^j(F).$$
Hence $F$ is consisits of exactly  $l$ different points that are fixed by $f^r$, and $F^r$  consists of at most $r l$ different points, also fixed by $f^r$.

We will define an homeomorphism $h: M \mapsto M$.
Consider a chart $(U_\alpha,\phi_\alpha)$ such that $\lambda \T_i \subset U_\alpha$.
Choose $\eta > 0$ so small such that for each $x \in F^r$ the solid ball $B(\phi_\alpha(x),\eta) \subset \mathbb R^m$ does not intersect $\phi_\alpha(U_\alpha \cap  \partial (\lambda \mathcal T))$, such  that the diameter of  $\mathbf B(x,\eta) := \phi^{-1} (B(x,\eta))$ is at most $\delta$, and such that the finite family $\{\mathbf B(x,\eta): \ x \in F^r\}$ is composed of pairwise disjoint sets.
We define  $h(y) = y$ if $y \not \in \bigcup_{x \in F^r} \mathbf B(x,\eta))$.
On the complement we define $h$ as follows. Fix $x \in F^r$, and
 consider polar/spherical  coordinates $(s,\theta): s \in [0,1], \theta \in \mathbb S^m$ to describe  the ball $B(x,\eta)$.   Finally
  use $\phi_\alpha^{-1}$ to pull back these
coordinates to $\mathbf B(x,\eta)$.

Fix a certain $k \ge 1$ sufficiently large (that will be chosen later), and define
$h : \mathbf B(x,\eta) \mapsto \mathbf B(x,\eta)$  by setting  $h(s,\theta) \defeq (s^{k},\theta)$. This construction for each point $x \in F^r$ completes the definition of the homeomorphism $h$.  Thus if $f$ is continuous then so is $$g: f \circ h,$$ and if $f$ is an homeomorphism, then $g$ is also an homeomorphism.

Note that for any $y \in M$ the map $h$ satisfies
$$dist(y,h(y)) < \delta \quad \mbox{ and } \quad dist(y,h^{-1}(y)) < \delta.$$
Since $g = f \circ h$ and $dist(h(y),y) <   \delta$ we have
$$\rho(f,g) < \e.$$
If besides $f$ is an homeomorphism, then $g^{-1} = h^{-1} \circ f^{-1}$ and
$$dist(g^{-1}(y),f^{-1}(y))  = dist(h^{-1}(z),z) < \delta < \e.$$
Thus
$$\rho(g^{-1} , f^{-1}) < \e.$$
We have proved that $g$ is an $\e$-perturbation of $f$ in $\C$ (resp.\ $\HM$). Now, to end the proof of the lemma, it is enough to  choose $k$ such that the set $\mathbf B(x,\eta)$ contains a periodic shrinking set of period that divides $r$ for each point $x$ of the finite set $  F^r$.

From the above construction, we have $h(0, \theta)= 0$. Therefore, $g(x) = g(0,\theta) = f(0,\theta) = f(x)$ for all $x \in F^r$; thus  each point $x$ of $F^r$  will
be fixed by the map $g^r$.

We claim that if we choose $k$ large enough, then $g$ has a shrinking set around the point $x \in T_i \cap F^r$ contained in $\mathbf B(x,\eta)$.
We consider the $(s,\theta)$ coordinates in the set  $\mathbf B(x,\eta)$.
Fix $\eta_1 \in (0,\eta)$. For each  $j=0,1,\dots,r-1$ choose a simplex $\overline I_j$ containing $f^j(x)$ such that
 $\overline I_j \subset \mathbf B(f^j(x),\eta_1)$.
 Next choose  $\eta_2 \in (0, \eta_1)$ small enough such that $f(\mathbf B(f^jx,\eta_2)) \subset I_{j+1}$ for $j=0,1,\dots,r-1$.
 In these coordinates $\mathbf B(x,\eta_i)$ is given by $\{(s,\theta): s \in [0,\eta_i]\}$.
Choose $k > 1$ so that $\eta_1^k < \eta_2$. For each $j = 1,\dots,r$ we have
 $$g(\overline I_j) \subset g(\mathbf B(f^jx,\eta_1)) = f \circ h (\mathbf B(f^jx,\eta_1)) \subset f(B^jx,\eta_2) \subset I_{j+1}.$$
 Thus the simplexes $\overline I_j$ are the desired shrinking sets.
This finishes the proof of Lemma \ref{lemmaP_qrIsDense}.
\end{proof}

\subsection{End of the proof of Theorem \ref{Theorem-main}}

\begin{proof}

In Corollary \ref{corollaryShrink} we have proved that
 $\O_f   \subset \overline{{\mathcal E}_f} =  \overline{\mbox{Per}}_f.$
Combining Theorems   \ref{theoremBigcapClosShr_q&SRBl-l}  and \ref{TheoremSRB-l=ClosureErgodic}, we deduce that
\begin{equation}
\nonumber \mbox{Per}_f \subset \bigcap_{\e >0}\bigcap_{q \geq1}\overline{\mbox{AShr}_{\e, q}\mbox{Per}_f} \subset  \mathcal O_f. \end{equation}
Since $\mathcal O_f$ is closed,   we conclude that ${\mathcal O}_f = \overline{{\mathcal E}_f} =  \overline{\mbox{Per}}_f $, as wanted.
\end{proof}


{\footnotesize \begin{thebibliography}{99}

\bibitem[AA]{AA} F.\ Abdenur and M.\ Andersson,
\emph{Ergodic theory of generic continuous maps},
Comm.\ Math.\ Phys.\ 318 (2013), no. 3, 831--855.

\bibitem[AHK]{AHK} E.\ Akin, M.\ Hurley, and J.A.\ Kennedy.
\emph{Dynamics of topologically generic homeomorphisms},
Mem.\ Amer.\ Math.\ Soc., 164(783): viii+130, 2003.

\bibitem[AP]{AP} V.S.\, Alpern, and S.\ Prasad, \emph{Typical dynamics of volume preserving homeomorphisms},
Volume 139 of Cambridge Tracts in Mathematics. Cambridge: Cambridge University Press, 2000

\bibitem[AH]{AH} N.\ Aoki, and K.\ Hiraide, \emph{Topological Theory of Dynamical Systems: Recent Advances},
North Holland 1994.

\bibitem[AM]{AM} A. Arbieto and C.A. Morales,
\emph{Expansivity of ergodic measures with positive entropy},
 ArXiv: 1110.5598 [math.DS], 2011


\bibitem[CE1]{CE1} E.\ Catsigeras and H.\ Enrich,
\emph{SRB-like measures for $C^0$ dynamics},
 Bull.\ Pol.\ Acad.\ Sci.\ Math. 59 (2011), no. 2, 151--164.

\bibitem[CE2]{CE2} E.\ Catsigeras and H.\ Enrich,
\emph{Equilibrium states and SRB-like measures of $C^1$-expanding maps of the circle},
Port.\ Math.\ 69 (2012), no. 3, 193--212.

\bibitem[CCE1]{CCE1} E.\ Catsigeras, M.\  Cerminara, and H.\ Enrich,
\emph{The Pesin entropy formula for $C^1$ diffeomorphisms with dominated splitting},
Ergodic Theory Dynam.\ Systems 35 (2015) 737--761.

\bibitem[CCE2]{CCE2}  E.\ Catsigeras, M.\  Cerminara, and H.\ Enrich,
\emph{Weak Pseudo-Physical Measures and Pesin's Entropy Formula for Anosov $C^1$ diffeomorphisms},
  Contemporary Mathematics 698 (2017) 69--89.

\bibitem[CT]{CT} E.\ Catsigeras, and S.\ Troubetzkoy,
\emph{Pseudo-physical measures for typical continuous maps of the interval},
arXiv:1705.10133v1

\bibitem[CS]{Coudene-S} Y. Coudene and B. Schapira, \emph{Generic measures for hyperbolic flows on non-compact spaces},  Israel Journ. Math. \ 179 (2010), 157--172.

\bibitem[H1]{H1}  M.\ Hurley,
\emph{On proofs of the $C^0$ general density theorem},
Proc.\  Amer.\  Math.\  Soc., 124 (1996) 1305--1309.

\bibitem[H2]{H2} M. Hurley,
\emph{Properties of attractors of generic homeomorphisms},
Ergodic Theory Dynam.\ Systems, 16 (1996) 1297--1310.

 \bibitem[GK]{Gelfert-K} K. Gelfert and D. Kwietniak, \emph{On density of ergodic measures and generic points}, Erg. Th. \& Dyn. Sys.
38 (2018) 1745--1767.

\bibitem[KMOP]{Oprocha&als}P.\ Koscielniak, M.\ Mazur, P.\ Oprocha and P.\ Pilarczyk,  \emph{Shadowing is generic--a continuous map case}, Discr.\ Contin.\ Dyn.\ Syst.\ 34 (2014) 3591--3609.

\bibitem[L]{L} J.M.\  Lee, \emph{Introduction to smooth manifolds} Springer Verlag, 2003.

\bibitem[M]{M} J.\. Munkres, \emph{Obstructions to the smoothing of piecewise-differentiable homeomorphisms}
Annals  Math.\ 72 (1960) 521--554.

\bibitem[O]{O} K.\ Odani, \emph{Generic homeomorphisms have the pseudo-orbit tracing property}, Proc.\ Amer.\ Math.\
Soc.\ 110 (1990) 281--284.

\bibitem[OU]{OU} J.C.\ Oxtoby, and S.M.\ Ulam, \emph{Measure-preserving homeomorphisms and metrical transitivity},
Ann.\ of Math. (2) 42 (1941) 874--920.

\bibitem[PP]{PP} S.Yu.\ Pilyugin, and O.B.\ Plamenevskaya, \emph{Shadowing is generic}, Topology and its Applications 97 (1999) 253--266.

\bibitem[W]{W} J.H.C.\ Whitehead,  \emph{On $C^1$-complexes},
 Annals  Math.\  41 (1940) 809--824.

\bibitem[Y]{Y} K.\ Yano, \emph{Generic homeomorphisms of $S^1$ have the pseudo-orbit tracing property},  J.\ Fac.\ Sci.\
Univ.\ Tokyo, Sect.\ IA Math. 34 (1987) 51--55.

\end{thebibliography}}
\end{document}